\documentclass{alggeom}

\usepackage{latexsym}
\usepackage{amsfonts}
\usepackage{amssymb}
\usepackage{rotating}
\usepackage{enumerate}
\usepackage{amsmath}
\usepackage{amscd}
\usepackage{amsthm}

\def\ints{{\mathbb Z}}
\def\rats{{\mathbb Q}}

\def\proj{{\mathbb P}}
\def\FF{{\mathbb F}}
\def\Frac{{\text{Frac}}}
\DeclareMathOperator{\Aut}{Aut}
\def\Gal{{\text{Gal}}}
\def\Ind{{\text{Ind}}}

\def\Spec{{\mbox{Spec }}}
\def\mc#1{\mathcal{#1}}
\def\ol#1{\overline{#1}}

\newtheorem{theorem}{Theorem}[section]
\newtheorem{prop}[theorem]{Proposition}
\newtheorem{lemma}[theorem]{Lemma}
\newtheorem{corollary}[theorem]{Corollary}
\theoremstyle{definition}
\newtheorem{remark}[theorem]{Remark}
\newtheorem{defn}[theorem]{Definition}
\numberwithin{equation}{section}
\newtheorem*{mainresult}{Theorem \ref{Tmain}}

\begin{document}

\title[Good reduction of three-point covers]{Good reduction of three-point Galois covers} 

\author{Andrew Obus}
\email{andrewobus@gmail.com}
\address{University of Virginia, 141 Cabell Drive, Charlottesville, VA
22904}
\classification{Primary 14H57, 11G32 14H30; Secondary 14H25, 14G20}
\keywords{three-point cover, good reduction, auxiliary cover, stable model}
\thanks{The author was supported by an NSF Postdoctoral Research
  Fellowship in the Mathematical Sciences, as well as NSF FRG grant DMS-1265290.  This paper was conceived and
written during a visit to the Max-Planck-Institut f\"{u}r Mathematik in Bonn.}

\begin{abstract}
Michel Raynaud gave a criterion for a three-point $G$-cover $f: Y \to X = \proj^1$, defined over a $p$-adic field $K$, to have good reduction.
In particular, if the order of a $p$-Sylow subgroup of $G$ is $p$, and the number of conjugacy classes of elements of order $p$ is greater than the 
absolute ramification index $e$ of $K$, then $f$ has potentially good reduction.  We give a different proof of this criterion, which extends to the 
case where $G$ has an arbitrarily large \emph{cyclic} $p$-Sylow subgroup, answering a question of Raynaud.  
We then use the criterion to give a family of examples 
of three-point covers with good reduction to characteristic $p$ and arbitrarily large $p$-Sylow subgroups.
\end{abstract}

\maketitle

\section{Introduction}\label{Sintro}
This paper is about three-point  
$G$-Galois covers of the projective line (that is, $G$-Galois covers $f:Y \to X = \proj^1$ that are \'{e}tale outside of $\{0, 1, \infty\} \subseteq X$).
In particular, we give a criterion for such covers to have potentially 
good reduction to characteristic $p$, when they are defined over a mixed characteristic $(0, p)$
discrete valuation field.  If $p \nmid |G|$, then it is known (\cite{sga1}) that the cover has potentially good reduction.
If $p$ divides the order of any ramification index, then the cover will have bad reduction 
(see \S\ref{Sstable} for what we specifically mean by this).  When $p$ divides $|G|$, but not any ramification index (the ``in-between" case), then 
it can be quite difficult to determine whether the cover has potentially good reduction.  
Our main theorem is the following (which is phrased equivalently, although slightly differently, in the body of the paper).

\begin{mainresult}
Let $G$ be a finite group with cyclic $p$-Sylow subgroup.
Let $K_0 = \Frac(W(k))$, where $k$ is an algebraically closed field of characteristic $p$.  Let $K/K_0$ be a finite extension of degree $e(K)$, 
where $e(K)$ is less than the number of conjugacy classes of order $p$ in $G$. 
If $f: Y \to X = \proj^1$ is a three-point $G$-cover defined over $K$ (as a $G$-cover), then $f$ has potentially good reduction,
realized over a tame extension $L/K$ of degree dividing the exponent of the center $Z(G)$ of $G$. 
In particular, if $Z(G)$ is trivial, then $f$ has good reduction.
\end{mainresult}
Note that, if $f$ is a cover satisfying the hypotheses of Theorem \ref{Tmain}, then the ramification indices of $f$ are all prime to $p$ 
(\cite[Lemma 4.2.13]{Ra:sp}).

The immediate motivation for Theorem \ref{Tmain} is the result \cite[Th\'{e}or\`{e}me 0]{Ra:sp} of Raynaud, which is Theorem \ref{Tmain}
in the case that the $p$-Sylow subgroup of $G$ has order exactly $p$.    
Raynaud asked (\cite[Question 6.2.1]{Ra:sp}) whether his theorem could be extended to the case where $G$ has a cyclic $p$-Sylow subgroup.
This was shown to be true
under a quite strong assumption in \cite[Corollary 5.11]{Ob:vc} (that the stable reduction of the cover has ``no new \'{e}tale tails"), 
and the current paper removes this assumption.

Raynaud's proof (in the $v_p(|G|) = 1$ case) was based on a study of the \emph{stable model}
of the cover $f$.  In particular, he showed that the action of the absolute Galois group $G_K$ on the stable reduction $\ol{f}$ of $f$
factors through a quotient of prime-to-$p$ order (i.e., that the ``wild monodromy" is trivial), and was able to conclude the good reduction from this.  
The author undertook a detailed study of this Galois action in \cite{Ob:vc}, but it is
still not clear whether the wild monodromy must be trivial for three-point covers as in Theorem \ref{Tmain}.  We are able to get around this obstacle 
by analyzing the arithmetic of certain extensions of discrete valuation rings attached to the stable model of $f$, see \S\ref{Splan}.

Of course, the greater background motivation for our result is to understand the tame fundamental group $\pi_1(U_k)^{\text{tame}}$ of  
$U_k := \proj^1_k \backslash \{0, 1, \infty\}$, where $k$ is algebraically closed of characteristic $p$.
If $K$ is an algebraic closure of $\Frac(W(k))$ and $U_{K} = \proj^1_{K} \backslash \{0,1, \infty\}$, 
then Grothendieck (\cite{sga1}) showed that there exists a well-defined surjection
$$pr: \pi_1(U_{K})^{p\text{-tame}} \to \pi_1(U_k)^{\text{tame}}$$
up to conjugation (here $\pi_1(U_{K})^{p\text{-tame}}$ is the inverse limit of the automorphism groups of $p$-tame \'{e}tale Galois covers of $U_K$,
that is, covers 
whose completion to a branched cover of $\proj^1_{K}$ has prime-to-$p$ branching indices).    
The surjection $pr$ is an isomorphism on the maximal prime-to-$p$ quotients.  Understanding the kernel of $pr$ (equivalently, 
$\pi_1(U_k)^{\text{tame}}$)
is equivalent to determining which $p$-tame three-point $G$-Galois covers of $\proj^1_{K}$ have good reduction.  
For a purely group-theoretic translation of our result, see Appendix \ref{Agroup}.

We mention that Ihara (\cite{Ih:cq}) has asked a related question, which he believes to have a positive answer.  
Namely, let $f$ be a three-point $G$-cover defined over $K$ as in Theorem \ref{Tmain}, with $G$ an \emph{arbitrary} finite group, and $e(K) < p-1$
(of course, $G$ must be generated by two elements).
Assume further that there is a $K$-rational point of $\proj^1 \backslash \{0, 1, \infty\}$ above which $f$ splits completely.  Does this imply that $f$ has
good reduction?  Ihara answers this question positively when $G$ is solvable (see \cite[\S4.3]{Ih:cq}).  He also shows that it is enough to show that
$f$ has potentially good reduction (\cite[Proposition 1]{Ih:cq}).  For a possible connection between this question and the techniques used in this paper,
see Remark \ref{Rihara}.

\subsection{Plan of the proof}\label{Splan}
Similarly to the proof of Raynaud, our proof of Theorem \ref{Tmain} depends upon understanding the stable reduction 
$\ol{f}: \ol{Y} \to \ol{X}$ of the $G$-cover $f$ to characteristic $p$.  We will recall the details we need in \S\ref{Sstable}.
If $f$ has bad reduction, there will be irreducible components of $\ol{X}$ above which the map $\ol{f}$ is inseparable, where information from the
original cover is seemingly lost.  But we are able to show that, when $f$ is a three-point cover with bad reduction, 
the reduction is \emph{multiplicative}.  That is, the 
action of certain subgroups $\ints/p^i \leq G$ on certain irreducible components of $\ol{Y}$ reduces to that of the multiplicative group 
scheme $\mu_{p^i}$.
This is done by generalizing a result of Wewers (\cite{We:mc}) on multiplicative reduction of $\ints/p \rtimes \ints/m$-covers ($p \nmid m$) to the case of
$\ints/p^i \rtimes \ints/m$-covers.  In \S\ref{Saux} we use the construction of the \emph{auxiliary cover} (originally by Raynaud in \cite{Ra:sp}, 
generalized by the author in \cite{Ob:fm2}), to reduce the case of general $G$ with cyclic $p$-Sylow subgroup to the case
of $\ints/p^i \rtimes \ints/m$.

We show independently in \S\ref{Sgoodred} (Proposition \ref{Pnotmult})
that a three-point $G$-cover (or even a $G$-cover of $\proj^1$ branched at arbitrarily
many equidistant points) defined over a ``small" field (i.e., a field $K$ such as in Theorem \ref{Tmain}) cannot possibly have multiplicative
reduction.  This uses results in \S\ref{Sval} on Galois extensions of mixed characteristic discrete valuation fields.  This allows us to
conclude Theorem \ref{Tmain}.  In \S\ref{Sexamples}, we apply Theorem \ref{Tmain} to give examples of infinite families of
covers with good reduction and arbitrarily large $p$-Sylow subgroups.  

\subsection{Notation}
For any group $G$ with a cyclic $p$-Sylow subgroup $P$, we set $m_G := |N_G(P)|/|Z_G(P)|$, the order of the normalizer of $P$ divided by the order of
the centralizer.  A \emph{$G$-Galois cover} (or just \emph{$G$-cover}) 
of projective, smooth, geometrically integral curves over a field $K$ is a finite map 
$f: Y \to X$ such that $K(Y)/K(X)$ is a $G$-Galois extension.  
If $f$ is a $G$-Galois cover, with $y \in Y$ such that $f(y) = x$, then the 
\emph{ramification index} of $y$ is equal to the \emph{branching index} of $x$, which is the ramification index of the extension
$\hat{\mc{O}}_{Y, y} \to \hat{\mc{O}}_{X, x}$ of complete local rings
(that is, the valuation of a uniformizer of $\hat{\mc{O}}_{X, x}$ in
$\hat{\mc{O}}_{Y, y}$.  If this index is greater than 1, then $y$ is called a \emph{ramification point} 
and $x$ is a \emph{branch point}. 

Let $f: Y \to X$ be any morphism of schemes and assume $H$ is a finite group with $H \hookrightarrow \Aut(Y/X)$.  If
$G$ is a finite group containing $H$, then we define the map $\Ind_H^G f: \Ind_H^G Y \to X$ by setting $\Ind_H^G Y$ to be a disjoint union of
$[G:H]$ copies of $Y$, indexed by the left cosets of $H$ in $G$, and applying $f$ to each copy.  
The group $G$ acts on $\Ind_H^G Y$, and the stabilizer
of each copy of $Y$ in $\Ind_H^G Y$ is a conjugate of $H$.

The valuation on any mixed characteristic discrete valuation field $K$ is written $v_K$, and is normalized so that the value group is $\ints$.
The symbol $\zeta_n$ always represents a primitive $n$th root of unity. For any mixed characteristic discrete valuation ring or field $R$, the
absolute ramification index (i.e., the valuation of $p$) is written $e(R)$. 

\begin{acknowledgements}
The main result of this paper is the first result I had originally wanted to prove for my Ph. D. thesis.  While I was not able to succeed
at the time, I was able to lay much of the essential groundwork for this paper.  I thank my thesis advisor, David Harbater, for 
invaluable help with this process.  I also thank Irene Bouw, Gerd Faltings, Michel Raynaud, and Kirsten Wickelgren for useful comments.
\end{acknowledgements}

\section{Extensions of discrete valuation fields}\label{Sval}
Let $L/K$ be a finite separable extension of complete discrete valuation fields (DVFs), 
with valuation rings $O_L/O_K$ and residue fields $\kappa_L/\kappa_K$.  
The extension $L/K$ (likewise $O_L/O_K$) 
is called \emph{unramified} if a uniformizer of $K$ is also a uniformizer of $L$, and the extension $\kappa_L/\kappa_K$ is separable.
If the extension is not unramified, it is called \emph{ramified}.  If $\kappa_L/\kappa_K$ is inseparable, the extension is called \emph{fiercely
ramified}.  If a uniformizer of $K$ is not a uniformizer of $L$, we will call the extension \emph{naively ramified}.  An extension that is not naively ramified 
is called \emph{weakly unramified} (it might be fiercely ramified or unramified).  Furthermore, 
if $L/K$ is a (not necessarily finite) extension of complete DVFs 
such that $K$ has perfect residue field, it is called \emph{unramified} if a uniformizer of $K$ is a uniformizer of $L$.

Now, let us assume that $L$ and $K$ have characteristic $(0, p)$, and that $L/K$ is a $\ints/p^n$-extension, for $n \geq 1$.  
The extension $L/K$ (likewise $O_L/O_K$) is said to be \emph{of $\mu_{p^n}$-type} if $\zeta_{p^n} \in K$ and $L/K$ is
a Kummer extension that can be generated by extracting a $p^n$th root 
of an element $a \in K$ such that $v_K(a) = 0$ and the reduction of $a$ is not a $p$th
power in $\kappa_K$.  
Such an extension is clearly fiercely ramified and weakly unramified,
with an inseparable residue field extension of degree $p^n$.  If $F \subseteq K$ is a complete DVF, then 
the extension $L/K$ (likewise $O_L/O_K$) 
is \emph{potentially of $\mu_{p^n}$-type with respect to $F$} if there exists a finite extension $F'/F$ of complete DVFs such that
the base change $L'/K' := (L \otimes_F F')/ (K \otimes_F F')$ of $L/K$ is a field extension of $\mu_{p^n}$-type.  

The terminology ``$\mu_{p^n}$-type" comes from the fact that the corresponding map $\Spec O_L \to \Spec O_K$ 
is a torsor under the group scheme $\mu_{p^n}$.

\begin{lemma}\label{Lmuptype}
Let $L/K$ be a $\ints/p^n$-extension of characteristic $(0, p)$ complete DVFs such that $\zeta_{p^n} \in K$, 
given as a Kummer extension by extracting a $p^n$th root of $a$, where
$v_K(a) = 0$.  For any complete DVF $F \subseteq K$, we have that $L/K$ is potentially of $\mu_{p^n}$-type with respect to $F$ if and only if it is of 
$\mu_{p^n}$-type.
\end{lemma}

\begin{proof}  The ``if" direction is trivial.  For the ``only if" direction, suppose $L/K$ is not of $\mu_{p^n}$-type, let $F'/F$ be a finite extension, and let 
$L'/K' = (L \otimes_F F')/ (K \otimes_F F')$.  Assume $L'/K'$ is a field extension.
Then $a$ reduces to a $p$th power $\ol{a}$ in the residue field $\kappa_{K}$ of $K$, and $\ol{a}$ is also a $p$th power in the residue field
$\kappa_{K'}$ of $K'$.  Now, $L' = K'(\sqrt[p^n]{a})$.  If we can also write  
$L' = K'(\sqrt[p^n]{b})$, with $v_K'(b) = 0$, then Kummer theory shows that $b = a^r c^{p^n}$ for some $r \in \ints$ and $c \in K'$.  
This means that $b$ reduces to a $p$th power in $\kappa_{K'}$, so $L'/K'$ is not of $\mu_{p^n}$-type.
\end{proof}

\begin{lemma}\label{Lpimpliespn}
Let $L/K$ be a $\ints/p^n$-extension of characteristic $(0, p)$ complete DVFs with $n \geq 1$, and let $M/K$ be the unique $\ints/p$-subextension. 
Suppose $F \subseteq K$ is a complete DVF with algebraically closed residue field such that $K$ is unramified over $F$.  Then
$L/K$ is potentially of $\mu_{p^n}$-type with respect to $F$ if and only if $M/K$ is potentially of $\mu_p$-type with respect to $F$.  
\end{lemma}

\begin{proof}
Throughout this proof, a base change of $K$ by a finite extension $F'/F$ (resp.\ $F''/F$, etc.) will be denoted $K'$ (resp.\ $K''$, etc.).  The
same holds for $L$ and $M$.

If $L/K$ is potentially of $\mu_{p^n}$-type with respect to $F$, then there is a finite extension $F'/F$  
such that $L' = K'(\sqrt[p^n]{a})$ with $v_{K'}(a) = 0$ and $a$ does not reduce to a $p$th
power in the residue field $\kappa_{K'}$ of $K'$.  Then $M'/K'$ is given by $M' = K'(\sqrt[p]{a})$, so
$M/K$ is potentially of $\mu_p$-type.  

Conversely, suppose a finite extension $F'/F$ yields $M'/K'$ given by
$M' = K'(\sqrt[p]{a})$, where $v_{K'}(a) = 0$ and $a$ does not reduce to a $p$th power in $\kappa_{K'}$.
Note that $K'$ is unramified over $F'$.  
If $F'' = F'(\zeta_{p^n})$, then $F''/F'$ is totally naively ramified (as $F'$ has algebraically closed residue field). 
So $K''/K'$ is also totally naively ramified, $v_K''(a) = 0$, and $a$ does not reduce to a $p$th power in the residue
field $\kappa_{K''} = \kappa_{K'}$ of $K''$.  Thus $M''/K''$ is given by $M'' = K''(\sqrt[p]{a})$, which is a nontrivial $\ints/p$-extension.
This means that $L''/K''$ is a $\ints/p^n$-extension given by
$L'' = K''(\sqrt[p^n]{b})$, where $b/a$ is a $p$th power in $K''$.  After a further totally naively ramified extension $F'''/F''$ (giving
a totally naively ramified extension $K'''/K''$), and possibly multiplying $b$ by a $p^n$th power in $K'''$, we may assume that $v_{K'''}(b) = 0$, 
and $a$ still does not reduce to a $p$th power in the residue field $\kappa_{K'''} = \kappa_{K'}$ of $K'''$.  
Since $b/a$ is a $p$th power in $K'''$, we have that $b$ does not reduce to 
a $p$th power in $\kappa_{K'''}$.  In particular, $L''' = K'''(\sqrt[p^n]{b})$, thus $L/K$ is potentially of multiplicative type.
\end{proof}

\begin{remark}\label{Ltossici}
One can give an alternate proof of Lemma \ref{Lpimpliespn} 
by using \cite[Lemma 3.2]{To:mup2} for the case $n=2$, and then induction for general $n$.
\end{remark}

\begin{prop}\label{Pval0}
Let $m > 1$ be prime to $p$, let $M/K$ be a $\ints/m$-extension of characteristic $(0, p)$ complete DVFs, 
and assume that $\zeta_p \in K$.
Let $a \in M^{\times} \backslash (M^{\times})^p$, and let $L = M(\sqrt[p]{a})$.  If
$L/K$ is $G$-Galois with $G$ nonabelian, then $p \, | \, v_M(a)$.
\end{prop}

\begin{proof}
By the Schur-Zassenhaus theorem, $G \cong \ints/p \rtimes \ints/m$, for some nontrivial action of $\ints/m$ on $\ints/p$.  Let $\sigma, c \in G$ be 
noncommuting elements of order $p$ and prime-to-$p$ respectively, such that $c \sigma  = \sigma^\nu c$, with $\nu \not \equiv 1 \pmod{p}$. Since 
$L/K$ is Galois, we know that $a$ and $c(a)$ both
yield the same Kummer extension, so $c(a) = a^dz^p$, where $z \in M$ and $p \nmid d$.  

Choose a $p$th root $\sqrt[p]{a}$ of $a$ in $L$, and let $\zeta$ be a $p$th root of unity such that $\sigma(\sqrt[p]{a}) = \zeta \sqrt[p]{a}$.  
Let $\alpha_{c}$ be such that $c(\sqrt[p]{a}) = \zeta^{\alpha_{c}}z\sqrt[p]{a}^d$.  Then 
$$c \sigma (\sqrt[p]{a}) = c(\zeta \sqrt[p]{a}) = \zeta c(\sqrt[p]{a}) = \zeta^{1+\alpha_{c}}z\sqrt[p]{a}^d,$$ whereas 
$$\sigma^\nu c(\sqrt[p]{a}) = \sigma^\nu (\zeta^{\alpha_{c}}z\sqrt[p]{a}^d) = \zeta^{\nu d+ \alpha_{c}} z \sqrt[p]{a}^d.$$
Thus $\nu d \equiv 1 \pmod{p}$. In particular, $d \not \equiv 1 \pmod{p}$.  Since $$v_M(a) \equiv v_M(c(a)) \equiv dv_M(a) \pmod{p},$$
we conclude that $p \, | \, v_M(a)$.
\end{proof}

\begin{lemma}\label{Llowramindex}
If $K$ is a characteristic $(0, p)$ complete DVF such that $e(K) < p-1$, then every ramified 
$\ints/p$-extension $L/K$ is naively ramified.
\end{lemma}

\begin{proof}
This is well known and follows, for example, from \cite[Proposition 3.3.2 and Th\'{e}or\`{e}me 3.3.3]{Ra:sg}.
\end{proof}

\begin{corollary}\label{Cgetalphap}
Let $m > 1$ be prime to $p$, and let $L/K$ be a $\ints/p \rtimes \ints/m$-extension of characteristic $(0, p)$ complete DVFs.  
Let $M$ be the intermediate field corresponding to the subgroup $\ints/p$.
Set $K'$ (resp.\  $L'$, $M'$) equal to $K(\zeta_p)$ (resp.\ $L(\zeta_p)$, $M(\zeta_p)).$  
Suppose that $e(M) < p-1$, and that $L'/K'$ is nonabelian.  
Then $L/M$ is not potentially of $\mu_p$-type with respect to any DVF $F \subseteq M$. 
\end{corollary}
 
\begin{proof}
By Lemma \ref{Llowramindex}, $L/M$ is naively ramified.
By Proposition \ref{Pval0} applied to $K' \subseteq M' \subseteq L'$,  we know that $L'/M'$ 
is given by extracting a $p$th root of some $a \in M'$ such that $p \, | \, v_{M'}(a)$.  After multiplying
by a $p$th power, we may assume $v_{M'}(a) = 0$.  Since $L/M$ is naively ramified and $M'/M$ is a prime-to-$p$ extension,
we have that $L'/M'$ is naively ramified, thus not of $\mu_p$-type.  By 
Lemma \ref{Lmuptype}, $L'/M'$ is not potentially of $\mu_p$-type with respect to any DVF $F' \subseteq M'$.
By definition, the same is true for $L/M$.
\end{proof} 

\begin{remark}\label{Rgetalphap}
In the situation of Corollary {\rm \ref{Cgetalphap}}, suppose further that the conjugation action of $\ints/m$ on $\ints/p$ is faithful.  One easily sees that
in order to check $L'/K'$ nonabelian, it suffices to have $M \not \subseteq K'$.
\end{remark}

\section{Stable reduction of covers}\label{Sstable}
Let $R$ be a mixed characteristic $(0, p)$ complete discrete valuation ring (DVR) with residue field $k$ and fraction field $K$.
We set $X \cong \proj^1_K$, and we fix a smooth model $X_R = \proj^1_R$ of $X$.  Let $G$ be a finite group with cyclic
$p$-Sylow group.
Let $f: Y \to X$ be a $G$-Galois cover defined over $K$, such that the branch points of $f$ are defined over $K$ and their specializations 
do not collide on the special fiber of $X_R$.  Assume that $f$ is branched at at least three points.
The cover $f$ may not have a smooth model over $R$, but it at least has a model over a finite extension of $R$ that is ``as close to
smooth as possible."
Specifically, by a theorem of Deligne and Mumford (\cite[Corollary 2.7]{DM:ir}), combined with work
of Raynaud (\cite[Appendice]{Ra:pg}, \cite[Proposition 2.4.5]{Ra:sp}) and Liu (\cite[Theorem 0.1]{Li:sr}), there is a
minimal finite extension $K^{st}/K$ with valuation ring $R^{st}$, and a unique model $f^{st}: Y^{st} \to X^{st}$ of $f_{K^{st}} := f \times_K K^{st}$
(called the \emph{stable model} of $f$) over $R^{st}$ such that

\begin{itemize}
\item The special fiber $\ol{Y}$ of $Y^{st}$ is semistable (i.e., has only ordinary double points for singularities). 
\item The ramification points of $f_{K^{st}}$ specialize to
\emph{distinct} smooth points of $\ol{Y}$.
\item $G$ acts on $Y^{st}$, and $X^{st} = Y^{st}/G$.
\item Any genus zero irreducible component of $\ol{Y}$ contains at least three
marked points (i.e., ramification points or points of intersection with the rest
of $\ol{Y}$).
\end{itemize}
If the last criterion is omitted, we say we have a \emph{semistable model} for $f$.
If we are working over a finite extension $K'/K^{st}$ with valuation ring $R'$, we will sometimes abuse
language and call $f^{st} \times_{R^{st}} R'$ the stable model of $f$.  This is justified because the special fiber of 
such a model is identical to that of the original stable model.

If $\ol{Y}$ is smooth, the cover $f: Y \to X$ is said to have \emph{potentially
good reduction}.   If $f$ does not have potentially good reduction, it is said to have \emph{bad reduction}.  In any case, the special fiber $\ol{f}:
\ol{Y} \to \ol{X}$ of the stable model is called the \emph{stable reduction} of $f$.  
One can view $X^{st}$ as a blowup of $X_R \times_R R^{st}$.  The strict transform of the special fiber of $X_R \times_R R^{st}$ in $X^{st}$ 
is called the \emph{original component}, and will be denoted $\ol{X}_0$.  

\subsection{Inertia Groups of the Stable Reduction}

The action of $G$ on $Y^{st}$ reduces to an action on the special fiber $\ol{Y}$.  By \cite[Lemme 6.3.3]{Ra:ab}, we know that
the inertia groups of the action of $G$ on $\ol{Y}$ at generic points of $\ol{Y}$ are $p$-groups.  
If $\ol{V}$ is an irreducible component of $\ol{Y}$, we will always write $I_{\ol{V}} \leq G$ for the inertia group of
the generic point of $\ol{V}$, and $D_{\ol{V}} \leq G$ for the decomposition group.

The inertia groups above a generic point of an irreducible component $\ol{W} \subset 
\ol{X}$ are conjugate cyclic $p$-groups.  If they have order
$p^i$, we call $\ol{W}$ a \emph{$p^i$-component}.  If $i = 0$, we call $\ol{W}$ an \emph{\'{e}tale
component}, and if $i > 0$, we call $\ol{W}$ an \emph{inseparable component}.  

As in \cite{Ra:sp}, we call an irreducible component $\ol{W} \subseteq \ol{X}$ a \emph{tail} if it is not the original component and intersects exactly 
one other irreducible component of $\ol{X}$.
Otherwise, it is called an \emph{interior component}.  
A tail of $\ol{X}$ is called \emph{primitive} if it contains the specialization of a branch
point of $f$.  Otherwise it is called \emph{new}.  An \'{e}tale component that is a tail is called an \emph{\'{e}tale tail}.

\begin{lemma}[\cite{Ob:vc}, Proposition 2.13] \label{Lcorrectspec}
If $x \in X$ is branched of prime-to-$p$ order, then $x$
specializes to an \'{e}tale component.
\end{lemma}

\begin{lemma}[\cite{Ra:sp}, Proposition 2.4.8]\label{Letaletail}
If $f$ has bad reduction and $\ol{W}$ is an \'{e}tale component of $\ol{X}$, then $\ol{W}$ is a tail.  In particular, the original component
is an inseparable component.
\end{lemma}

\begin{remark}\label{Rmultred}
We will see in \S\ref{Smultred} that, even for fixed $i$, 
not all $p^i$-components are created equal!  Indeed, one wishes to keep track of some of the information that is lost when reducing
to characteristic $p$.  For a reasonably full picture of what gets lost, the standard technique is to use ``deformation data" as developed, 
e.g., in \cite[\S1]{He:ht}, \cite[\S1.3]{We:mc}, and \cite[Construction 2.4]{Ob:vc}.  However, in this paper, we will only need to keep track of
whether certain inseparable components are ``multiplicative" or not, see \S\ref{Smultred}.  Our proof of Theorem \ref{Tmain} is based on
showing that if $f$ has bad reduction, then under the assumptions of Theorem \ref{Tmain},
the original component must be both multiplicative and not multiplicative, a contradiction.
In \S\ref{Stame}, we give an explicit sufficient criterion for the original component to be multiplicative, when $G \cong \ints/p^s \rtimes \ints/m$. 
\end{remark}

\subsection{Multiplicative reduction}\label{Smultred}
Maintain the notation previously introduced in \S\ref{Sstable}. 
Suppose that $f: Y \to X = \proj^1$ has bad reduction, and let $f^{ss}: Y^{ss} \to X^{ss}$ be a semistable model for $f$, defined over $R$, with 
special fiber $\ol{f}^{ss}: \ol{Y}^{ss} \to \ol{X}^{ss}$.  
Let $\ol{V}$ be an irreducible component of $\ol{Y}^{ss}$ above an inseparable component $\ol{W}$ of $\ol{X}^{ss}$, let $I_{\ol{V}}$ 
be its inertia group, and let $D_{\ol{V}}$ be its decomposition group.  Since $I_{\ol{V}}$ is normal in $D_{\ol{V}}$, it follows that 
$D_{\ol{V}}$ has a normal subgroup of order $p$.  Thus, by \cite[Corollary 2.4]{Ob:vc}, the maximal normal prime-to-$p$
subgroup $N$ of $D_{\ol{V}}$ is such
that $D_{\ol{V}}/N$ has normal $p$-Sylow subgroup $P$ of some order $p^j$.

If $\eta$ is the generic point of $\ol{V}/N$, then $P$ acts on the complete DVR $B := \hat{\mc{O}}_{Y_R/N, \eta}$.  Let
$C$ be the fixed ring $B^P$.  Since $P$ acts trivially on the residue field $k(\ol{V}/N)$ of $B$, we have that $B/C$ is a totally
ramified $P$-extension of complete DVRs.  If $B/C$ is potentially of $\mu_{p^j}$-type with respect to $\Frac(R)$, 
then the model $f^{ss}$ is said to have \emph{multiplicative reduction above $\ol{W}$} (in the language of Remark \ref{Rmultred}, the component
$\ol{W}$ is multiplicative).
As mentioned at the beginning of \S\ref{Sval}, this means that $B/C$ is a $\mu_{p^j}$-torsor.

\subsection{Vanishing cycles formula}\label{Svanishing}
The original version of the \emph{vanishing cycles formula} below is due to Raynaud in \cite{Ra:sp}.  The generalized version below is a 
special case of \cite[Theorem 3.14]{Ob:vc}, which will be vital for Proposition \ref{Pmulttype}.  The idea is to relate the genera and
ramification behavior of certain irreducible components of the stable reduction of $f$ to the ramification behavior of $f$ itself.
First we need a definition.  Maintain the notation
from the beginning of \S\ref{Sstable}.

\begin{defn}[(cf. \cite{Ob:fm1}, Definition 4.10)]\label{Draminvariant}
Consider an \'{e}tale tail $\ol{X}_b$ of $\ol{X}$.  Suppose $\ol{X}_b$ intersects the rest of $\ol{X}$ at $\ol{x}_b$.
Let $\ol{Y}_b$ be a component of $\ol{Y}$ lying above $\ol{X}_b$, and let $\ol{y}_b$ be a point
lying above $\ol{x}_b$.  Then the \emph{effective ramification invariant} $\sigma_b$ of $\ol{X}_b$ is the conductor of higher ramification
for the extension $\hat{\mc{O}}_{\ol{Y}_b, \ol{y}_b}/\hat{\mc{O}}_{\ol{X}_b, \ol{x}_b}$ of complete DVRs.  That is, 
$$\sigma_b = \sup_{i \in \rats_{\geq 0}}(H^i \neq id),$$
where $H$ is the Galois group of $\hat{\mc{O}}_{\ol{Y}_b, \ol{y}_b}/\hat{\mc{O}}_{\ol{X}_b, \ol{x}_b}$ and $H^i$ is the 
filtration for the \emph{upper numbering} (see \cite[IV \S3]{Se:lf}).
We also set $m_b$ equal to the prime-to-$p$ part of $|H|$.  
\end{defn}

\begin{lemma}[(\cite{Ob:vc}, Lemma 2.20, Lemma 4.2)]\label{Lramdenominator}
The effective ramification invariants $\sigma_b$ are positive and lie in $\frac{1}{m_G}\ints$.  Furthermore, if $\ol{X}_b$ is a new tail,
then $\sigma_b \geq 1 + 1/m_G$. 
\end{lemma}

\begin{theorem}[(Vanishing cycles formula, \cite{Ob:vc}, Theorem 3.14, Corollary 3.15)]\label{Tvancycles}
Let $f: Y \to X \cong \proj^1$ be a $G$-Galois cover with bad reduction branched at $r$ points as in this section, where $G$ has a cyclic
$p$-Sylow subgroup.  Let $B_{\text{new}}$ be an indexing set for the new \'{e}tale tails and let
$B_{\text{prim}}$ be an indexing set for the primitive \'{e}tale tails.  Let $\sigma_b$ be the effective ramification invariant in Definition 
\ref{Draminvariant}.
Then we have the formula 

\begin{equation}\label{Evancycles} 
r-2 = \sum_{b \in B_{\text{new}}} (\sigma_b - 1) + \sum_{b \in B_{\text{prim}}} \sigma_b.
\end{equation}
\end{theorem}

\subsection{Covers of multiplicative type}\label{Stame}
Maintain the notation of this section, and assume for the duration of \S\ref{Stame} that $G = \ints/p^s \rtimes \ints/m$, where $m$ is prime to $p$,
the action of $\ints/m$ on $\ints/p^s$ is faithful, and $s \geq 1$.  Note that $m_G = m$.  Fix elements
$c \in G$ of order $m$ and $\sigma \in G$ of order $p$.  Then we can define a character $\ol{\chi}: \ints/m \to \FF_p^{\times}$ such that
$c^i \sigma c^{-i} = \sigma^{\ol{\chi}(i)}$ for all $i \in \ints/m$.  By \cite[Lemma 2.1]{Ob:vc}, the character $\ol{\chi}$ is injective; in particular, 
$m | p-1$.  We lift $\ol{\chi}$ to a character $\chi: \ints/m \to K$ such that $\chi(i)$ is the unique $m$th root of unity whose reduction is $\ol{\chi}(i)$.

Let $f: Y \to X = \proj^1$ be a $G$-cover defined over $K$ branched at $r$ $K$-points $x_1, \ldots, x_r$, with $r \geq 3$.  Assume that the
branching indices are all prime to $p$.  For notational purposes, choose a coordinate for $X$ such that $x_i \neq \infty$ for $1 \leq i \leq r$.  
Throughout \S\ref{Stame}, we \emph{allow} the specializations of the branch points to collide on the special fiber of $X_R$.  
There is then a unique stable model $\ol{X}^{adm}$ of $X_R$ separating the specializations of the branch points, such that each irreducible
component of the special fiber $\ol{X}^{adm}$ of $\ol{X}$ contains at least three marked points (i.e., branch points and singular points of $\ol{X}^{adm}$).
In the case of a three-point cover, $\ol{X}^{adm}$ is smooth.

The cover $f$ is the composition of a $\ints/m$-cover $g: Z \to X$ given (birationally) by an equation of the form

\begin{equation}\label{Emulttype}
z^m = \prod_{i=1}^r (x - x_i)^{a_i}
\end{equation}
with an \'{e}tale cover $Y \to Z$ of degree $p^n$. Here $z$ is chosen so that $c^*z = \chi(1)z$.
We may assume that, for $1 \leq i \leq r$, we have $0 \leq a_i \leq m$.    Since $g$ is unramified at $\infty$, we have
$m | \sum_{i=1}^r a_i$.  If $\sum_{i=1}^r a_i = m$, then $f$ is said to be of \emph{multiplicative type} (cf. \cite[\S1]{We:mc}).

Let $f^{ss}: Y^{ss} \to X^{ss}$ be a semistable model for $f$ with special fiber $\ol{f}^{ss}: \ol{Y}^{ss} \to \ol{X}^{ss}$.  
Since a semistable model must separate the specializations of branch points,  
we have $\ol{X}^{ss} \supseteq \ol{X}^{adm}$.  If $\ol{X}_b$ is an \'{e}tale tail of $\ol{X}^{ss}$ that intersects an inseparable component
(not automatic, as $f^{ss}$ is not necessarily the stable model),
then one can define $\sigma_b$ and $m_b$ for $\ol{X}_b$ as in Definition \ref{Draminvariant}.

The following proposition is analogous to \cite[Proposition 1.8 (i)]{We:mc}.
\begin{prop}\label{Pinvarianteq}
Given $f^{ss}$ as above, suppose $\ol{X}_b$ is an \'{e}tale tail of $\ol{X}^{ss}$ containing the specialization $\ol{x}_i$ of $x_i$ and no other 
specializations of branch points.  Suppose further that $\ol{X}_b$ intersects an inseparable component of $\ol{X}^{ss}$.
Then $$\langle \sigma_b \rangle = \frac{a_i}{m},$$ where $\langle \cdot \rangle$ means the fractional part.
\end{prop}

\begin{proof}
Let $u$ be a coordinate on $\ol{X}_b$ such that $u = \infty$ corresponds to the intersection point $\ol{x}_b$ with the rest of $\ol{X}^{ss}$.
If $H = \ints/p^a \rtimes \ints/m_b$ is the decomposition group of an irreducible component of $\ol{Y}^{ss}$ above $\ol{X}_b$, then
$\sigma_b$ is defined by a jump in the upper numbering of an $H$-Galois extension of $k[[u^{-1}]]$, 
and replacing $\sigma_b$ by any other jump does not change its fractional part (\cite[Theorem 1.1]{OP:wt}).  So it suffices to prove the 
theorem for the \emph{first} positive upper jump $\sigma_b'$ of the Galois extension in Definition \ref{Draminvariant} (i.e., the smallest number $i >0$ 
such that $H^i \supsetneq H^{i+\epsilon}$ for all $\epsilon > 0$).  Since the upper numbering is invariant under taking quotients 
(\cite[IV, Proposition 14]{Se:lf}), taking the quotient of the cover $f^{ss}$ by the unique subgroup of order $p^{a-1}$ in $G$ does not affect the 
quantities in the proposition, and we may assume $a = 1$. 

Let $\ol{Y}_b = \ol{Y}^{ss} \times_{\ol{X}^{ss}} \ol{X}_b$, which is a disjoint union of $H$-covers of $\ol{X}_b$.  
Let $\ol{V}_b \subseteq \ol{Y}_b$ be a disjoint union of $m/m_b$ of these covers on which a group $\Gamma \cong \ints/p \rtimes \ints/m$ 
with $H \subseteq \Gamma \subseteq G$ acts transitively. 
If we let $\tau_i$ be the point where $\ol{X}_b$ intersects the rest of $\ol{X}^{ss}$, then the quotient $\ol{V}_b/(\ints/p) \to \ol{X}_b$
is of type $(\ol{x}_i, \tau_i, a_i, m-a_i)$ with respect to $\ol{\chi}$ in the sense of \cite[p.\ 116]{We:mc}.  
Thus, the proof of \cite[Proposition 1.8 (i)]{We:mc} carries through verbatim to prove the proposition.
%
%
\end{proof}

The term ``multiplicative type" is closely related to multiplicative reduction, as the next proposition shows.
\begin{prop}\label{Pmulttypered}
If $f: Y \to X$ is of multiplicative type with $G = \ints/p^s \rtimes \ints/m$, 
then any semistable model $f^{ss}$ has multiplicative reduction (\S\ref{Smultred}) above every irreducible component of $\ol{X}^{adm}$.
\end{prop}

\begin{proof}
If $s = 1$, then this follows from \cite[Proposition 1.3]{We:mc} and \cite[p.\ 190]{Ra:pg} (see proof of \cite[Corollary 1.5]{We:mc}).
If $s \geq 1$, then $f$ has a quotient $\ints/p \rtimes \ints/m$-cover $h: W \to X$, which is also of multiplicative type.  Thus any semistable model
$h^{ss}$ of $h$ has multiplicative reduction above all of $\ol{X}^{adm}$.  This implies that each irreducible component of $\ol{X}^{adm}$
is a $p^s$-component for $f$.  Let $K^{ss}/K$ be a finite extension over which $f^{ss}$ is defined.  Since $K^{ss}$ has algebraically closed residue field, 
Lemma \ref{Lpimpliespn} implies that $f^{ss}$ has multiplicative reduction over all of $\ol{X}^{adm}$.
\end{proof}

\section{The auxiliary cover}\label{Saux}
Our goal in this section is to simplify the group-theoretical structure of a $G$-cover $f$ by replacing it with the so-called \emph{auxiliary cover}, which is
originally an idea of Raynaud.  This will be a cover whose geometry is similar to $f$, but which has more branch points and a simpler 
Galois group.  In particular, the auxiliary cover will have a quotient, called the \emph{strong auxiliary cover}, whose Galois group is isomorphic
to $\ints/p^s \rtimes \ints/m$, for some $s$ and some $m$ that is prime to $p$.  The main goal of this section is to show (Proposition 
\ref{Pmulttype}) that the strong auxiliary cover of a three-point $G$-cover with bad reduction is of multiplicative type when $G$ has a cyclic $p$-Sylow
subgroup.  
This will have the consequence (Corollary \ref{Cmultred}) that the original cover has multiplicative reduction over the original component.

Specifically, let $G$ be a finite group with cyclic $p$-Sylow group.
Assume that $f: Y \to X = \proj^1$ is a $G$-cover 
defined over $K$ as in \S\ref{Sstable} with bad reduction, such that the specializations of the branch points do not collide on the special fiber of $X_R$.
By the construction given at the beginning of \cite[\S7]{Ob:fm2} (which is originally based on \cite[\S3.2]{Ra:sp}), one can construct 
(over some finite extension $K'/K$) an auxiliary cover $f^{aux}: Y^{aux} \to X^{aux} := X$ with a semistable model 
$(f^{aux})^{ss}: (Y^{aux})^{ss} \to (X^{aux})^{ss}$ and special fiber $\ol{f}^{aux}: \ol{Y}^{aux} \to \ol{X}^{aux}$.
We do not repeat the construction here, but we summarize the important properties, which all follow from \cite[\S7]{Ob:fm2}.  In particular, 
properties (ii), (iii) and (v) of Proposition \ref{Paux} below show that $f^{aux}$ has similar geometry to $f$, whereas property (vi) 
shows that the group theory of $f^{aux}$ is rather simple.

\begin{prop}\label{Paux}
Assume that $p$ does not divide any branch index of $f$. 
\begin{enumerate}[(i)]
\item The cover $f^{aux}$ is a $G^{aux}$-Galois cover, where $G^{aux} \leq G$.
\item We have $(X^{aux})^{ss} = X^{st}$ and $\ol{X}^{aux} = \ol{X}$.  
\item There exists an \'{e}tale neighborhood $Z$ (relative to $X^{st}$) 
of the union $\ol{U}$ of the inseparable components of $\ol{X}$, such that
the cover $f^{st} \times_{X^{st}} Z$ is isomorphic to $\Ind_{G^{aux}}^G (f^{aux})^{ss} \times_{X^{st}} Z$.  
\item The cover $f^{aux}$ has a branch point $x_b$ of index $m_b$ for each \'{e}tale tail $\ol{X}_b$ of $\ol{X}$ such that
$m_b > 1$ (Definition \ref{Draminvariant}).
If $\ol{X}_b$ is a primitive tail, then $x_b$ is the corresponding point branched in $f$.  If $\ol{X}_b$ is a new tail, then
$x_b$ specializes to a smooth point of $\ol{X}_b$.  These points comprise the entire branch locus of $f^{aux}$ (see Lemma \ref{Lcorrectspec}).
\item If $\ol{X}_b$ is an \'{e}tale tail of $\ol{X}$, 
and if $\ol{V}_b$ is an irreducible component of $\ol{Y}^{aux}$ above $\ol{X}_b$, then $\ol{V}_b \to \ol{X}_b$
is generically \'{e}tale.  If the effective ramification invariant $\sigma_b^{aux}$ of $\ol{V}_b \to \ol{X}_b$ is defined as in Definition \ref{Draminvariant},
then $\sigma_b^{aux} = \sigma_b$.
\item If $N$ is the maximal prime-to-$p$ normal subgroup of $G^{aux}$, then
$G^{aux}/N \cong \ints/p^s \rtimes \ints/m_{G^{aux}}$, where $s \geq 1$ and the action of $\ints/m_{G^{aux}}$ on $\ints/p^s$ is faithful. 
\end{enumerate}
\end{prop}

\begin{remark}\label{Raux}
In the context of Proposition \ref{Paux} (iii), $\ol{U}$ is a tree, and thus has trivial fundamental group.  So $\ol{f} \times_{\ol{X}} \ol{U}$ is isomorphic to
$\Ind_{G^{aux}}^G \ol{f}^{aux} \times_{\ol{X}} \ol{U}$.  In particular, the inseparable components of $\ol{X}$ for the auxiliary
cover are the same as for the original cover.
\end{remark}

In light of Proposition \ref{Paux} (vi), write $G^{str} := G^{aux}/N \cong \ints/p^s \rtimes \ints/m_{G^{aux}}$, 
where $N$ is the maximal prime-to-$p$ normal subgroup of $G^{aux}$. 
The canonical $G^{str}$-quotient gover of $f^{aux}$ is called the \emph{strong auxiliary cover}, and is written
$f^{str}: Y^{str} \to X^{str} := X^{aux} = X$.  We also write $(f^{str})^{ss}: (Y^{str})^{ss} \to (X^{str})^{ss}$ and $\ol{f}^{str}: \ol{Y}^{str} \to \ol{X}^{str}$ 
for $(f^{aux})^{ss}/N$ and its reduction, respectively.
Since the higher ramification filtration for the upper numbering is invariant under taking quotients 
(\cite[IV, Proposition 14]{Se:lf}), the effective ramification invariants
for the strong auxiliary cover are the same as those for the auxiliary cover, which (by Proposition \ref{Paux} (v)) 
are the same as those for the original cover.

\begin{prop}\label{Pmulttype}
The strong auxiliary cover $f^{str}$ of a \emph{three-point} $G$-cover $f$ with prime-to-$p$ branching is of multiplicative type (\S\ref{Stame}).
\end{prop}

\begin{proof}
For each \'{e}tale tail $\ol{X}_b$ of $\ol{X}^{str}$, let $\sigma_b$ be its effective ramification invariant, which is equal to the corresponding 
effective ramification invariant of $f$.  The vanishing cycles formula (\ref{Evancycles}) shows that  
$$1 = \sum_{b \in B_{\text{new}}} (\sigma_b - 1) + \sum_{b \in B_{\text{prim}}} \sigma_b,$$
where $b$ ranges over the \'{e}tale tails of $\ol{X}^{str}$ (here ``new" and ``primitive" refer to the corresponding tails of $f$).  
By Lemma \ref{Lramdenominator}, each term in the sums on the right hand side is positive.  
By Lemma \ref{Lcorrectspec}, there are three primitive \'{e}tale tails, 
thus namely more than one \'{e}tale tail, so each term is also strictly between $0$ and $1$.
Thus we conclude both that $\sum_b {\langle \sigma_b \rangle} = 1$ and that $m_b > 1$ for each \'{e}tale tail.  

By Proposition \ref{Paux} (iv), each \'{e}tale tail contains the specialization of exactly
one branch point of $f^{str}$.  Also, $f^{str}$ has a $\ints/m_{G^{aux}}$-quotient cover given by an equation in the form of 
(\ref{Emulttype}).  Since $\sum_b {\langle \sigma_b \rangle} = 1$, Proposition \ref{Pinvarianteq} shows that $f^{str}$ is of multiplicative type.
\end{proof}

\begin{corollary}\label{Cmultred}
Let $f: Y \to X$ be a three-point $G$-cover with bad reduction, such that all branching indices are prime to $p$.  
Then the stable model of $f$ has multiplicative reduction over the original component. 
\end{corollary}

\begin{proof}
By Propositions \ref{Pmulttype} and \ref{Pmulttypered}, any semistable model of the strong auxiliary cover $f^{str}$ has multiplicative 
reduction over the original component.  Since $f^{str}$ is a prime-to-$p$ quotient of
the auxiliary cover $f^{aux}$, we have that any semistable model of $f^{aux}$ also has multiplicative reduction over the original component.  By 
Proposition \ref{Paux} (iii), over an \'{e}tale neighborhood of the original component, the cover $f$ is isomorphic to a disjoint union 
of copies of $f^{aux}$.  Thus the stable model of $f$ also has multiplicative reduction over the original component.
\end{proof}

\begin{remark}\label{Rwewers}
If $v_p(|G|) = 1$, then Corollary \ref{Cmultred} follows rather easily from \cite[\S1.4 and the proof of Corollary 1.5]{We:mc}.  Indeed, 
the vanishing cycles formula and the auxiliary cover construction were already known for $v_p(|G|) = 1$ at the time of \cite{We:mc}.
The key new ingredients for the general case are the auxiliary cover construction and the vanishing cycles formula when $G$ has an arbitrarily large
cyclic $p$-Sylow subgroup.
\end{remark}

\begin{lemma}\label{Lnoncomm}
Let $f: Y \to X$ be a $G$-cover branched at $r \geq 3$ points as in this section, with bad reduction, such that all branching indices are prime to $p$.
Let $\ol{V}$ be an irreducible component above the original component, with decomposition group $D_{\ol{V}} \subseteq G$.
Then $m_{D_{\ol{V}}} > 1$.
\end{lemma}

\begin{proof}
By Remark \ref{Raux}, $\ol{V}$ can be viewed as an irreducible component of $\ol{Y}^{aux}$ with decomposition group isomorphic to
$D_{\ol{V}}$, so it suffices to prove the lemma for the auxiliary cover $f^{aux}$.  Let $\ol{V}^{str}$ be the image of $\ol{V}$ 
under the canonical map $(Y^{aux})^{ss} \to (Y^{str})^{ss}$.  
The decomposition group $D_{\ol{V}^{str}}$ is a quotient of $D_{\ol{V}}$ by a prime-to-$p$ subgroup.
It follows easily that $m_{D_{\ol{V}^{str}}} = m_{D_{\ol{V}}}$.  Furthermore, since the strong auxiliary cover has Galois group
$G^{str} \cong \ints/p^s \rtimes \ints/m_{G^{aux}}$ (for some $s$), with faithful conjugation action of $\ints/m_{G^{aux}}$ on $\ints/p^s$, 
any subgroup $\Gamma$ of $G^{str}$ that has order divisible by $p$, but is not a $p$-group, has $m_{\Gamma} > 1$.  
Thus it suffices to show that $D_{\ol{V}^{str}}$ is not a $p$-group.  

Set $m = m_{G^{aux}}$.  Then $f^{str}$ has a $\ints/m$-quotient cover 
given birationally by 
\begin{equation}\label{Etamecover}
z^m = \prod_{i=1}^{\rho} (x-x_i)^{a_i},
\end{equation} 
where $0 < a_i < m$.  It suffices to prove that this cover does not split completely over the original component, that is, that the reduction 
\begin{equation}\label{Ered}
\prod_{i=1}^{\rho} (x - \ol{x}_i)^{a_i}
\end{equation} 
of the right hand side of (\ref{Etamecover}) modulo the maximal ideal is not an $m$th power in $k(x)$.

Since $f$ is branched in $r$ different points, assumed to reduce to
pairwise distinct points of $\proj^1_k$, we know that there are at least $r$ different residue classes represented among the $\ol{x}_i$.
By the vanishing cycles formula (\ref{Evancycles}), Lemma \ref{Lramdenominator}, and Proposition
\ref{Pinvarianteq}, we have $\left(\sum_{i=1}^{\rho} a_i \right)/m \leq r-2$.  Since at least $r$ residue classes are represented among the 
$\ol{x}_i$, there is at least one residue class, comprised of
$\ol{x}_{i_1}, \ldots, \ol{x}_{i_{\nu}}$, such that $0 < \sum_{j=1}^{\nu} a_{i_j} < m$.  
This means that (\ref{Ered}) is not an $m$th power, and we are done.
\end{proof}

\section{Good reduction}\label{Sgoodred}

In this section, $k$ is an algebraically closed field of characteristic $p$, and $K_0 = \Frac(W(k))$.  The main result here is 
Proposition \ref{Pnotmult}.

\begin{prop}\label{Pnotmult}
Let $G$ be a finite group with nontrivial cyclic $p$-Sylow subgroup.
Let $K/K_0$ be a finite extension such that $e(K)< (p-1)/m_G$, and let $f: Y \to X = \proj^1$ be a $G$-cover defined over $K$, 
branched at distinct $K$-points $\{x_1, \ldots, x_r\}$, with $r \geq 3$.  
Let $R$ be the valuation ring of $K$, and suppose that there is a smooth model $X_R$ for $X$ 
such that the specializations of the branch points do not collide on the special fiber.
If $f$ has bad reduction, then the stable model of $f$ does not have multiplicative reduction over the original component.
\end{prop}

\begin{proof}
Suppose $f$ has bad reduction.
Let $f_R: Y_R \to X_R$ be the normalization of $X_R$ in $K(Y)$, and let $\ol{f}: \ol{Y} \to \ol{X}$ be the special fiber of $f_R$.
Pick an irreducible component $\ol{V}$ of $\ol{Y}$, and let $\eta_{\ol{V}}$ and $\eta_{\ol{X}}$ be the respective generic points  
(note that $\ol{V} \to \ol{X}$ is automatically finite).
Then $\hat{\mc{O}}_{Y_R, \eta_{\ol{V}}}/\hat{\mc{O}}_{X_R, \eta_{\ol{X}}}$ is a Galois extension of mixed characteristic $(0, p)$ complete
DVRs containing $R$.  
Let $\Delta$ be the Galois group of this extension, and $I$ the inertia group.  Then $I$ is normal in $\Delta$, and the $p$-Sylow subgroup 
of $I$ (which is nontrivial by Lemma \ref{Letaletail}) is characteristic in $I$, so it is normal in $\Delta$.  
By \cite[Corollary 2.4]{Ob:vc}, there is a prime-to-$p$ normal subgroup $N$ of $\Delta$ such that $\Delta/N \cong P \rtimes \ints/m_{\Delta}$, 
where $P$ is a $p$-group and $\ints/m_{\Delta}$ acts faithfully on $P$.  

Let $P' \leq P$ be the (nontrivial) $p$-Sylow subgroup of $I/N$, and let $P'' < P'$ be the unique subgroup of index $p$.
Write $$E \supseteq D \supseteq C \supseteq B \supseteq A,$$
where
$$E := (\hat{\mc{O}}_{Y_R, \eta_{\ol{V}}})^N, \ D := E^{P''}, \ 
C := E^{P'}, \ B := E^{I/N}, \ A := E^{\Delta/N} = \hat{\mc{O}}_{X_R, \eta_{\ol{X}}}.$$  
Since $m_{\Delta} | m_G$ and $e(B) = e(A) = e(K)$, we have $$e(C) \leq e(B)m_{\Delta} \leq e(K) m_G < p-1.$$   

Furthermore, let $R'/R$ be a finite extension containing $\zeta_p$ over which $f$ attains a stable model.  Let $E'$ (resp.\ $D'$, $C'$, $B'$, $A'$)
be the normalization of the compositum of $E$ (resp.\ $D$, $C$, $B$, $A$) and $R'$ viewed as lying in some fixed algebraic closure of $A$. 
Then $$E' \supseteq D' \supseteq C' \supseteq B' \supseteq A'$$ are mixed characteristic complete DVR's, and 
$\Gal(E'/A')$ is a quotient of the decomposition group of an irreducible component of the stable reduction of $f$ over the original component.
If $f$ has multiplicative reduction over the original component, then $E'/C'$ is potentially of $\mu_{[E':C']}$-type with respect to $\Frac(R')$.  
To obtain a contradiction, we will show that $D'/C'$ is a $\ints/p$-extension that is not potentially of $\mu_p$-type with respect to $\Frac(R')$ 
(Lemma \ref{Lpimpliespn} gives the contradiction).  To do this, it suffices to show that $D/C$ is not potentially
of $\mu_p$-type with respect to $\Frac(R)$, and that $D \not \subseteq C'$.  
 
Now, Lemma \ref{Lnoncomm} shows that $m_{\Gal(E'/A')} > 1$, so $\Gal(C'/A')$ has an element $g$ with prime-to-$p$ order $m > 1$.
Then $g$ acts nontrivially on $C$, so $D/C^{\langle g \rangle}$ is a $\ints/p \rtimes \ints/m$-extension, with faithful action of
$\ints/m$ on $\ints/p$.  Furthermore, $C^{\langle g \rangle}(\zeta_p)
\not \supseteq C$, as $g$ fixes $\zeta_p$ but acts nontrivially on $C$.  Since $e(C) < p-1$, Corollary \ref{Cgetalphap} and 
Remark \ref{Rgetalphap} applied
to $D/C^{\langle g \rangle}$ show that $D/C$ is not potentially of $\mu_p$-type with respect to $\Frac(R)$.  Furthermore,
$C'/C^{\langle g \rangle}$ is abelian, whereas $D/C^{\langle g \rangle}$ is not, thus $D \not \subseteq C'$.  We are done.
\end{proof}

\begin{theorem}\label{Tmain}
Let $G$ be a finite group with cyclic $p$-Sylow subgroup.
Let $K/K_0$ be a finite extension such that $e(K) < (p-1)/m_G$, and let $f: Y \to X = \proj^1$ be a three-point $G$-cover defined 
(as a $G$-cover) over $K$.  Then $f$ has potentially good reduction,
realized over a tame extension $L/K$ of degree dividing the exponent of the center $Z(G)$ of $G$. 
In particular, if $Z(G)$ is trivial, then $f$ has good reduction.
\end{theorem}

\begin{proof}
Since $e(K) < (p-1)/m_G$, \cite[Lemme 4.2.13]{Ra:sp} shows that the branching indices of $f$ are all prime-to-$p$. If $f$ has bad reduction, then
Corollary \ref{Cmultred} shows that the stable model of $f$ has multiplicative reduction over the original component.
However, by Proposition \ref{Pnotmult}, this is impossible.  Thus $f$ must have potentially good reduction.

By \cite[Proposition 4.1.2]{Ra:sp}, if $L/K$ is the minimal extension such that $f \times_K L$ has good reduction, then $\Gal(L/K)$ is a subgroup
of $Z(G)$.  If $p$ divides $|Z(G)|$, then \cite[Corollary 2.4]{Ob:vc} shows that $G$ has a quotient that is a nontrivial $p$-group.  
Thus $f$ must be branched with ramification indices divisible by $p$, which we have already seen is impossible. So $p \nmid \Gal(L/K)$.
This means that $\Gal(L/K)$ is cyclic, so its degree divides the exponent of $Z(G)$.
\end{proof}

\begin{remark}
Since all $p$-Sylow subgroups of $G$ are conjugate to each other, so are all subgroups of order $p$.  If $Q$ is such a subgroup, 
then the $G$-conjugacy class of a nontrivial element $q \in Q$ contains $m_G$ elements of $Q$.  Thus there are $(p-1)/m_G$ different 
conjugacy classes of elements of order $p$ in $G$, and the form of Theorem \ref{Tmain} stated above is equivalent to that in the introduction.
\end{remark}

\begin{remark}\label{Rraynaud1}
Raynaud asked (\cite[Question 6.2.2]{Ra:sp}) if Theorem \ref{Tmain} might hold when $G$ has a $p$-Sylow group of order $p$ and 
$r - 2 < p/m_G$, where $r$ is the number of branch points and all branch points are equidistant.  
For instance, should Theorem \ref{Tmain} hold for four-point covers with equidistant branch locus when $m_G \neq p-1$?
Unfortunately, Corollary \ref{Cmultred} need not hold for covers with more than three branch points, 
so it looks as though other techniques must be used for such covers.
\end{remark}

\begin{remark}\label{Rgeneral}
It is not hard to show (with essentially the same proof) that if $X$ is  
any smooth curve over $K$ of genus $g_X$ with good reduction, if $f: Y \to X$ is branched at $r$ $K$-points specializing to distinct points 
on a smooth model of $X$, and if $2g_X - 2 + r > 0$, then the analog of 
Proposition \ref{Pnotmult} holds as long as the analog of Lemma \ref{Lnoncomm} holds.  
Thus, if $f: Y \to X$ is as in Theorem \ref{Tmain}, then $f$ will have good reduction if one can
show that $f$ having bad reduction would imply the analogs of Corollary \ref{Cmultred} and Lemma \ref{Lnoncomm}.  It would be interesting
to see if this is the case when $g_X = 1$ and $r = 1$, which in many ways is analogous to the case $g_X = 0$ and $r = 3$ 
(for instance, the fundamental groups of $X \times_K \ol{K}$ are the same in both cases, and by \cite[Th\'{e}or\`{e}me 5.1.5]{Ra:sp}, 
Theorem \ref{Tmain} is true in both cases when $v_p(|G|) = 1$ and $G$ has trivial center).
\end{remark}

\begin{remark}\label{Rihara}
In the vein of Remark \ref{Rgeneral}, 
one method of attacking Ihara's question from the introduction would be to define a notion of multiplicative reduction for
$G$-covers ($G$ arbitrary), and show that all three-point covers with bad reduction have this type of reduction.  One might then be able to show
that if the cover is defined over a small field and has a split point, then it cannot have this type of reduction. 
This notion of multiplicative reduction should somehow reflect the fact the three-point covers are rigid (cf.\ \cite[Introduction]{We:def}).
\end{remark}

\section{Examples}\label{Sexamples}

Let $k$ be an algebraically closed field of characteristic $p$, and let $K_0 = \Frac(W(k))$.
We exhibit a family of three-point covers defined over $K_0$ that have good reduction, generalizing 
\cite[Example 5.12]{Ob:vc}. 
In particular, fix integers $n \geq 1$ and $m \geq 2$, and let $p \equiv 1 \pmod{m}$ be a prime not equal to $m+1$.
For any $m$, $n$, and $p$ as above, we construct (for infinitely many $q$) a $G := PGL_m(q)$-cover over $K_0$ with good reduction,
where $G$ has $p$-Sylow subgroup of size at least $p^n$.  Note that $G$ has trivial center.

Let $q$ be a prime power such that 
\begin{equation}\label{Egoodred}
q^m \equiv 1 \pmod{p^n}, \text{ but } q^j \not \equiv 1 \pmod{p^n} \text{ for } 1 \leq j < m.
\end{equation}
Since $p \equiv 1 \pmod{m}$, the group $(\ints/p^n)^{\times}$ has elements of exact order $m$, and (\ref{Egoodred}) has integer solutions $q$.
Then Dirichlet's theorem shows that there are infinitely many prime power solutions $q$ 
(in fact, infinitely many prime solutions).

Assume $q = \ell^d$ is a prime power satisfying (\ref{Egoodred}).  
We know by \cite[II, Satz 7.3]{Hu:eg}, along with an examination of the order of $G$, that $G = PGL_m(q)$ has a cyclic $p$-Sylow subgroup 
of order $p^{v_p(q^m - 1)} \geq p^n$. 
The same construction as in \cite[Example 5.12]{Ob:vc} now works to construct a three-point $G$-cover defined over $K_0$ with 
potentially good reduction to characteristic $p$.  Since it is brief, we include it here.

Consider $H := GL_m(q) \rtimes \ints/2$, where the $\ints/2$-action is inverse-transpose.
In \cite[II, Proposition 6.4 and Theorem 6.5]{MM:ig}, a rigid class vector
$(\tilde{C}_0, \tilde{C}_1, \tilde{C}_{\infty})$ is exhibited for 
$H/ \{\pm 1\}$, where $\tilde{C}_0$ has order $2$, $\tilde{C}_1$ has order $4$, and $\tilde{C}_{\infty}$ has order
$(q-1)\ell^a$ for some $a$ (this is because the characteristic polynomial for the elements of $\tilde{C}_{\infty}$ has eigenvalues of 
order $q-1$).  Since $p$ does not divide the order of any of the ramification indices, this triple is rational over $K_0$, so the corresponding
$H/ \{\pm 1\}$-cover is defined over $K_0$.  Thus there is a quotient $G \rtimes \ints/2$-cover $h: Y \to \proj^1$ defined over $K_0$. 

Let $X \to \proj^1$ be the quotient cover of $h: Y \to \proj^1$ corresponding to the group $G$.  
Then $X \to \proj^1$ is a cyclic cover of degree 2, branched at $0$ and $1$.
This means that $X \cong \proj^1$, and $Y \to X$ is branched at three points (the two points above 
$\infty$, and the unique point above $1$).  So we have constructed a three-point $G$-cover $f: Y \to X \cong \proj^1$ defined over $K_0$, 
such that all branch points have prime-to-$p$ branching index.  Since the $e(K_0)$ is $1$, and since
$m = m_G < p-1$, we have $e < (p-1)/m_G$.  We conclude that $f$ has good reduction using Theorem \ref{Tmain}.

\appendix

\section{Group-theoretic translation of Theorem \ref{Tmain}}\label{Agroup}

Let $k$ be an algebraically closed field of characteristic $p$, let $K_0 = \Frac(W(k))$, and let $K$ be an algebraic closure of $K_0$.
For any field $L$, write $U_L = \proj^1 \backslash \{0, 1, \infty\}$.  We then have the fundamental exact sequence
\begin{equation}\label{Efes}
1 \to \pi_1(U_K)^{p\text{-tame}} \to \pi_1(U_{K_0})^{p\text{-tame}} \to \Gal(K/K_0) \to 1,
\end{equation}
where ``$p$-tame" has the meaning from the introduction.  A finite three-point $G$-cover $f$ of $\proj^1$ defined over $K$ with prime-to-$p$
branching indices (which we will call a \emph{$p$-tame cover}) corresponds to the open normal subgroup $N_f$ of 
$\pi_1(U_K)^{p\text{-tame}}$.  One knows that $\pi_1(U_K)^{p\text{-tame}}/N_f \cong G$.  
The \emph{field of moduli} of a $G$-cover $f$ is the minimal extension $M_f/K_0$ contained in $K$
such that the action of $\Gal(K/M_f)$ on $\pi_1(U_K)^{p\text{-tame}}$ by conjugation preserves $N_f$ and descends to an inner automorphism on 
$\pi_1(U_K)^{p\text{-tame}}/N_f$ (this means that $\Gal(K/M_f)$ preserves the isomorphism class of $f$ when acting on the coefficients of the
equations for $f$).  We note that the conjugation action
is \emph{a priori} only an outer action, but that the field of moduli does not depend on which representative action is taken.
Since $K_0$ and all finite extensions have cohomological dimension $1$, \cite[Proposition 2.5]{CH:hf}
shows that $f$ can be defined over a given finite extension $L/K_0$ exactly when $L \supseteq M_f$.

Also, from the introduction, we have a surjection
$$pr: \pi_1(U_K)^{p\text{-tame}} \to \pi_1(U_k)^{\text{tame}}.$$
An element $\sigma \in \pi_1(U_{K})^{p\text{-tame}}$ is determined (up to conjugation) by its action on the system of $p$-tame three-point 
covers of $\proj^1_K$, and it is in the kernel of $pr$ precisely when it acts trivially on every such cover that has good reduction.  In other
words, it must be contained in the corresponding normal subgroup for each such cover.  As a 
consequence of Theorem \ref{Tmain} and the discussion above, we obtain the following group-theoretic restatement of the main theorem.

\begin{corollary}\label{Cmain}
Let $\sigma \in \pi_1(U_{K})^{p\text{-tame}}$.  If $\sigma \in \ker(pr)$, then $\sigma$ is contained in every finite index normal subgroup
$N$ of $\pi_1(U_{K})^{p\text{-tame}}$ such that both
\begin{enumerate}[(i)]
\item $\pi_1(U_{K})^{p\text{-tame}}/N \cong G$ has a cyclic (possibly trivial) $p$-Sylow subgroup.
\item There exists a subgroup $\Gamma \leq \Gal(K/K_0)$ of index less than $(p-1)/m_G$ such that
the conjugation action of $\Gamma$ on $\pi_1(U_{K})^{p\text{-tame}}$ from (\ref{Efes}) fixes $N$ and descends to an inner automorphism
of $\pi_1(U_{K})^{p\text{-tame}}/N$ (Note: this is automatic if $\pi_1(U_{K})^{p\text{-tame}}/N$ is prime to $p$, as all such covers can be 
defined over $K_0$).
\end{enumerate}
\end{corollary} 

\bibliographystyle{amsalpha}
\bibliography{literatur}

\end{document}